\documentclass{amsart}
\usepackage{amsfonts,amscd}

\newtheorem{claim}{}[section]
\newtheorem{theorem}[claim]{Theorem}
\newtheorem{lemma}[claim]{Lemma}
\newtheorem{proposition}[claim]{Proposition}
\newtheorem{corollary}[claim]{Corollary}

\newtheorem{definition}[claim]{Definition}

\def\proclaim #1. #2\par{\medbreak
\noindent{\bf#1.\enspace}{\sl#2}\par\medbreak}
\makeatother

\DeclareMathOperator{\Cdb}{{\mathbb C}}

\DeclareMathOperator{\Ddb}{{\mathbb D}}
\DeclareMathOperator{\Tdb}{{\mathbb T}}

\begin{document}

\title[Rigged modules over dual operator algebras]{Rigged modules I: modules over dual operator algebras and the Picard group}

\author{David P. Blecher and Upasana Kashyap}
\address{Department of Mathematics, University of Houston,
Houston, TX  77204-3008}
\email{dblecher@math.uh.edu}

\address{Department of STEM, Regis College,
Weston, MA 02493}
\email{upasana.kashyap@regiscollege.edu}
\date{\today}


\begin{abstract}
In a previous paper we generalized the theory of $W^*$-modules to the setting of modules over nonselfadjoint dual
operator algebras, obtaining the class of weak$^*$-rigged modules.  At that time  we 
promised a forthcoming paper devoted to other aspects of the theory. We fulfill this promise in the present work and its sequel ``Rigged modules II",  giving many new results about  weak$^*$-rigged modules and their tensor products.  We also discuss the Picard group of weak* closed subalgebras of a commutative algebra.  For example, we compute the weak Picard group of $H^{\infty}(\mathbb{D})$, and prove that for a weak* closed function algebra $A$, the weak Picard
group is a semidirect
product of the automorphism group of $A$,  and the subgroup 
consisting of symmetric equivalence bimodules.

\end{abstract}

\maketitle
\section{Introduction and notation}    The most important class of modules over a $C^*$-algebra $M$ are 
the Hilbert $C^*$-modules: the modules possessing an $M$-valued inner product satisfying 
the natural list of axioms (see \cite{Lance} or \cite[Chapter 8]{DBbook}).  
A {\em $W^*$-module} is a Hilbert $C^*$-module over a von Neumann algebra which is selfdual (that is,
it satisfies the module variant of the fact from Hilbert space theory
that every continuous functional is given by inner product with a fixed vector, see e.g.\ \cite{Pas,Bsd}), or equivalently
which has a Banach space predual (see e.g.\ \cite[Corollary 3.5]{BM}).    A {\em dual operator algebra} is a unital weak* closed
algebra of operators on a Hilbert space (which is not necessarily selfadjoint).
One may think of a dual operator algebra as a nonselfadjoint  analogue of a 
von Neumann algebra.  The {\em weak$^*$-rigged} or  $w^*$-{\em rigged modules}, introduced in \cite{BK2} (see also \cite{BKraus,UK}), are a generalization of $W^*$-modules to the setting of modules over a (nonselfadjoint) dual
operator algebra. 
In \cite{BK2}  we generalized basic aspects of the theory of $W^*$-modules, and this may be seen also as the weak* variant of the theory of rigged modules from \cite{DB1} (see also \cite{BMP}). 
In that paper we promised that some other aspects of the theory of weak$^*$-rigged modules 
would be presented elsewhere.  Since that time is now long overdue, and since there has been some recent interest in rigged modules and related 
objects (see e.g.\ \cite{Mes} or the survey \cite{Elef} of Eleftherakis' work),  we 
follow through on our promise here and in the sequel \cite{DBr}.   

The present paper and its sequel consists of several topics and results about 
weak$^*$-rigged modules, mainly concerning  their tensor products.   For example following the route in \cite{BJ} we study the Picard group of weak* closed subalgebras of a commutative algebra.
For a weak* closed function algebra $A$, the weak Picard
group $Pic_{w}(A)$ is a semidirect
product of  Aut$(A)$, the automorphism group of $A$,  and the subgroup of $Pic_{w}(A)$
consisting of symmetric equivalence bimodules. In particular, we show that the weak Picard group of $H^{\infty}(\mathbb{D})$ is isomorphic to the group of conformal automorphisms of the disk.

We will use the notation from \cite{BK1,BK2,UK1}, and perspectives from \cite{DB2}. We will assume that the
reader is familiar with basic notions from operator space theory which may be found in any current text on that subject
(such as \cite{ERbook,DBbook}). The reader may
consult \cite{DBbook} as a reference for any other unexplained terms
here. We also assume that the reader is familiar with basic Banach space and operator space duality principles,
e.g., the Krein-Smulian Theorem (see e.g., Section 1.4, 1.6, Appendix A.2 in \cite{DBbook}). 
We often abbreviate `weak$^*$' to `$w^*$'.    We use the letters $H,K$ for Hilbert spaces.  
By a nonselfadjoint analogue of Sakai's theorem (see e.g.\ Section 2.7 in \cite{DBbook}), 
a dual operator algebra $M$ is characterized as a unital operator algebra which is also a dual operator
space.   By a {\em normal morphism} we shall always mean a
unital weak* continuous completely contractive homomorphism
on a dual operator algebra.
 A concrete dual operator $M$-$N$-bimodule is a weak* closed subspace $X$ of $B(K, H)$ such that 
$\theta(M)X \pi(N) \subset X$, where $\theta$ and $\pi$ are normal morphism representations of $M$ and $N$ on $H$ and $K$ respectively. An abstract dual operator $M$-$N$-bimodule is defined to be a nondegenerate operator $M$-$N$-bimodule $X$, which is also a dual operator space, such that the module actions are separately weak* continuous. Such spaces can be represented completely isometrically as concrete dual operator bimodules, and in fact this can be done under even weaker hypotheses (see e.g.\ \cite{DBbook,BM,ER}) and similarly for one-sided modules (the case $M$ or $N$ equals $\Cdb$). We use standard notation for module mapping spaces; e.g.\ $CB(X, N)_N$ (resp.\ $CB^\sigma(X, N)_N$) are the completely bounded (resp.\ and weak* continuous) right $N$-module maps from $X$ to $N$.  We often use the {\em normal module Haagerup tensor product} $\otimes^{\sigma h}_{M}$, and its universal property from \cite{EP}, which loosely says that it `linearizes'
completely contractive $M$-balanced separately weak* continuous bilinear maps (balanced means that 
$u(xa,y) = u(x,ay)$ for $a \in M$).  We assume that the reader is familiar with the notation and facts about this tensor product from \cite[Section 2]{BK1}.   For any operator space $X$ we write $C_n(X)$ for the column space of $n \times 1$ matrices with entries in $X$, with its canonical norm from operator space theory.  

\begin{definition} \cite{BK2} \label{wrig} Suppose that $Y$ is a dual operator
space and a right module over a dual operator algebra $M$. Suppose
that there exists a net of positive integers $(n(\alpha))$, and
$w^*$-continuous completely contractive $M$-module maps
$\phi_{\alpha} : Y \to C_{n(\alpha)}(M)$ and $\psi_{\alpha} :
C_{n(\alpha)}(M) \to Y$, with $\psi_{\alpha}( \phi_{\alpha}(y))  \to
y$ in the weak* topology on $Y$, for all $y \in Y$.   Then we say
that $Y$ is a {\em right $w^*$-rigged module} (or {\em weak$^*$-rigged module}) over $M$.
\end{definition}

As on p.\ 348 of \cite{BK2}, the operator space structure of a 
$w^*$-rigged module $Y$ over  $M$ is determined by
$$\Vert [y_{ij}] \Vert_{M_n(Y)}
= \sup_\alpha \; \Vert [\phi_{\alpha}(y_{ij})] \Vert , \qquad
[y_{ij}] \in M_n(Y). $$

The following seems not to have been proved  in the development in Section 2 and the 
start of Section 3 in \cite{BK2}
but seemingly assumed there.  

\begin{lemma} \label{l}   A right  $w^*$-rigged module over 
a dual operator algebra $M$ is a dual operator $M$-module.  
\end{lemma}

\begin{proof}    Let $Y$ be the $w^*$-rigged module,  and let $\phi_\alpha, \psi_\alpha, n_\alpha$ be as in 
Definition \ref{wrig}.  We need to show that the module action 
$Y \times M \to Y$  is separately weak* continuous. 
Given a bounded net $m_t \to m$ weak* in $M$, and $y \in Y$, suppose that a subnet
$y m_{t_{\nu}} \to y'$ weak* in $Y$.   Then $\phi_\alpha (\psi_\alpha (y m_{t_{\nu}})) = \phi_\alpha (\psi_\alpha (y) m_{t_{\nu}})$ weak* converges both
to $\phi_\alpha (\psi_\alpha (y )m) = \phi_\alpha (\psi_\alpha (ym))$ and 
$\phi_\alpha (\psi_\alpha (y'))$ with $\nu$, for every $\alpha$.   Hence
$y' = ym$,  so that by topology $y m_t \to ym$ weak*.
So the map  $m \mapsto ym$ is  weak* continuous by the Krein-Smulian theorem.  
That each $m \in M$
acts weak* continuously on $Y$ follows from e.g.\ \cite[Corollary 4.7.7]{DBbook}.
\end{proof}

Every right $w^*$-rigged module $Y$ over $M$ gives rise to a canonical left $w^*$-rigged $M$-module
$\tilde{Y}$, and a pairing $(\cdot, \cdot) : \tilde{Y} \times Y \to M$ (see \cite{BK2}).
Indeed $\tilde{Y}$ turns out to be completely isometric to $CB^{\sigma}(Y,M)_{M}$ as dual operator 
$M$-modules, together with its canonical pairing with $Y$.   We also have $\widetilde{\tilde{Y}} = Y$.
The morphisms between $w^*$-rigged $M$-modules are the {\em adjointable} $M$-module maps;
these are the $M$-module maps $T : Y_1 \to Y_2$ for which there exists a
$S : \tilde{Y_2} \to \tilde{Y_1}$ with $(x,T(y)) = (S(x),y)$ for all $x \in \tilde{Y_2}, y \in Y_1$.
These turn out to coincide with the 
weak* continuous completely bounded  $M$-module maps (see \cite[Proposition 3.4]{BK2}).   
We often write $\mathbb{B}(Z,W)$ for the weak* continuous completely bounded  $M$-module maps
from a  $w^*$-rigged $M$-module $Z$ into a dual operator $M$-module $W$, with as
usual $\mathbb{B}(Z) = \mathbb{B}(Z,Z)$.

In  \cite{BK2}, Section 4 we gave several equivalent definitions of
$w^*$-rigged modules (an additional note that seems to be
needed to connect the definitions may be found mentioned in the proof of 
\cite[Theorem 2.3]{DBr}).  
    From some of these it is clear that every 
weak* Morita equivalence $N$-$M$-bimodule $Y$ in the following sense, is
a $w^*$-rigged right $M$-module and a $w^*$-rigged left $N$-module,
and its `dual module' $\tilde{Y}$ and pairing  $(\cdot, \cdot) : \tilde{Y} \times Y \to M$
may be taken to be the $X$ below, and pairing corresponding to the first $\cong$ below.

\begin{definition} \label{bk1} \label{w*Mor}  We say that two dual operator algebras $M$ and $N$ are {\em weak*
Morita equivalent}, if there exist a dual operator $M$-$N$-bimodule $X$, and a
dual operator $N$-$M$-bimodule $Y$, such that  $M \cong X \otimes^{\sigma h}_{N} Y$
as dual operator $M$-bimodules (that is, completely
isometrically, $w^{*}$-homeomorphically, and also as
$M$-bimodules), and similarly 
 $N \cong Y \otimes^{\sigma h}_{M} X$
as dual operator $N$-bimodules.
We call $(M, N , X, Y)$ a {\em weak* Morita context} in this case. 
\end{definition}

\section{Rigged and weak$^*$-rigged modules and their tensor product}

\subsection{Interior tensor product of weak* rigged modules}

Suppose that $Y$ is a right $w^*$-rigged module over a dual operator algebra $M$
and, that $Z$ is a right $w^*$-rigged module over 
a dual operator algebra $N$, and $\theta :M \to \mathbb{B}(Z)$
is a normal morphism. Because $Z$ is a left 
operator module $\mathbb{B}(Z)$-module
(see p.\ 349 in  \cite{BK2}), $Z$ becomes an essential
left dual operator module over $M$ under 
the action $m \cdot z = \theta(m) z$. In this case we say $Z$ is a right $M$-$N$-{\em correspondence}.
We form the normal module Haagerup tensor product  $Y \otimes^{\sigma h}_{M} Z$
which we also write as $Y \otimes_{\theta} Z$.  By 3.3 in \cite{BK2} this a right $w^*$-rigged module
over $N$. We call $Y \otimes_{\theta} Z$ the {\em interior tensor product} of $w^*$-rigged modules.
By 3.3 in \cite{BK2}  we have $\widetilde{Y \otimes_{\theta} Z} \cong \tilde{Z} \otimes_{\theta} \tilde{Y}$ with 
$N$-valued pairing
$$( w \otimes x, y \otimes z)_{N} = (w, \theta((x,y)_M)z)_N$$
of $\tilde{Z} \otimes_{\theta} \tilde{Y}$ with $Y \otimes^{\sigma h}_{M} Z$.

The $w^*$-rigged interior  tensor product is associative. That is, if $A$, $B$, and $C$ are dual
operator algebras, if $X$ is a right $w^*$-rigged module over $A$, $Y$ is a right $w^*$-rigged module over $B$,
 and $Z$ is a right $w^*$-rigged module over $C$, and if $ \theta :A \to \mathbb{B}(Y)$ and $ \rho :B \to \mathbb{B}(Z)$
 are normal morphisms, then $(X \otimes_{\theta} Y) \otimes_{\rho} Z \cong X \otimes_{\theta} (Y \otimes_{\rho} Z)$
completely isometrically and $w^*$-homeomorphically.  
This follows from the associativity of the normal module Haagerup tensor product (see Proposition 2.9 in \cite{BK1}).

Let $Y$ be a right $w^*$-rigged module over a dual operator algebra $M$.
If $N$ is a $W^*$-algebra and if $Z$ is a right $W^*$-module over $N$ and 
$\theta : M \to \mathbb{B}(Z)$ is a normal morphism,
 then as we said earlier, $Y \otimes_{\theta} Z$
is a right $w^*$-rigged module over $N$. 
Since $N$ is  a $W^*$-algebra it follows
from Theorem 2.5 in \cite{BK2} that $Y \otimes_{\theta} Z$ is a right $W^*$-module over $N$.
An important special case  of this is  the {\em $W^*$-dilation}.
If $Z = N$ is a  von Neumann algebra containing (a weak* homeomorphic
completely isometrically isomorphic copy of) $M$, with the same identity element,
and $\theta : M \to N$ is the inclusion, 
then $Y \otimes_{\theta} N$ is called the $W^*$-dilation of $Y$ (see \cite{UK1,BK1,BK2}).
The following is an application of the $W^*$-dilation to Morita equivalence.

\begin{theorem} \label{dimoq}  
Suppose that $Y$ is a right $w^*$-rigged module over a dual
operator algebra $M$. Suppose that $N$ is a von Neumann algebra containing $M$ as a weak* closed 
subalgebra (with the same identity), and that ${\mathcal R}$ is the weak* closed ideal in $N$ generated by the range of the pairing $\tilde{Y} \times Y
\to M$.    Let $\theta : M \to N$ be the inclusion, which also 
induces a map $\pi : M \to B_{\mathcal R}({\mathcal R}) \cong {\mathcal R}$ via the canonical left action. 
Then $Z = Y \otimes_{\theta} N = Y \otimes_{\pi} {\mathcal R}$ is a von Neumann algebraic Morita equivalence bimodule (in the sense of Rieffel {\rm \cite{Rief}}), 
implementing a von Neumann algebraic Morita equivalence between ${\mathcal R}$ 
and $\mathbb{B}(Z) =  Y \otimes_{\pi} {\mathcal R} \otimes_{\pi} \tilde{Y}$.
In particular, if $W^*(M)$ is  a von Neumann algebra generated by (a weak* homeomorphic
completely isometrically isomorphic copy of) $M$, and if the range of the pairing $\tilde{Y} \times Y
\to M$ is weak* dense in $M$, then $W^*(M)$
is  Morita equivalent (in the sense of Rieffel) to the $W^*$-algebra 
$\mathbb{B}(Y \otimes_{\theta} W^*(M)) = Y \otimes_{\theta} W^*(M) \otimes_{\theta} \tilde{Y}$. 
\end{theorem}
\begin{proof}    By the lines above the theorem, $Z = Y \otimes_{\theta} 
N$ is  a right $W^*$-module over $N$.      That $Y \otimes_{\theta} N = Y \otimes_{\pi} {\mathcal R}$
may be seen for example by \cite[Theorem 3.5]{BK2} (see Lemma \ref{sdrel} below), since 
any weak* continuous left $M$-module map $\tilde{Y} \to N$ necessarily takes values in the ideal ${\mathcal R}$
since terms of form $(x,y) x'$ are weak* dense in $\tilde{Y}$, for $x,x' \in \tilde{Y}, y \in Y$. 
By Corollary 8.5.5 in \cite{DBbook}, $\mathbb{B}(Z)$ is a $W^*$-algebra.  Consider the canonical pairing
$$\tilde{Z} \times Z \cong (N \otimes_{\theta} \tilde{Y}) \times (Y \otimes_{\theta} N) \to N.$$
The $w^*$-closure of the range of this pairing is a weak* closed two sided ideal in $N$.
It is easy to see that this ideal is ${\mathcal R}$, since the terms of the form $a (xy) b$
are contained in the range of the above pairing for all $(x,y) \in \tilde{Y} \times Y$
and $a, b \in N$.     By  \cite{Rief} (or 8.5.14 in \cite{DBbook}),
${\mathcal R}$ and  $\mathbb{B}(Z)$ are Morita equivalent in the sense
of Rieffel.  That $$\mathbb{B}(Z) = Y \otimes^{\sigma h}_M ({\mathcal R} \otimes^{\sigma h}_{\mathcal R} 
{\mathcal R})
 \otimes^{\sigma h}_M \tilde{Y} = 
  Y \otimes_{\theta} {\mathcal R}  \otimes_{\theta} \tilde{Y}$$
follows from the second paragraph on p.\ 357 of \cite{BK2}. 

Since $W^*(M)$ is a $W^*$-algebra generated by $M$,
and it is well known that the weak* closed ideals of $W^*$-algebras are selfadjoint, these together imply
that the weak* closure of the range of the above pairing is all of $W^*(M)$. By 8.5.14 in \cite{DBbook},
$W^*(M)$ and  $\mathbb{B}(Y \otimes_{\theta} W^*(M))$ are Morita equivalent in the sense
of Rieffel.     
\end{proof}

 \subsection{Functorial properties} \label{functo}

  Note that 
$M_{m,n}(\mathbb{B}(Y_1,Y_2))$ has two natural norms: the operator space one,  
coming from $CB(Y_1,M_{m,n}(Y_2))$, or the  
norm 
coming from $CB(C_n(Y_1),C_m(Y_2))$ via the identification
of a matrix $[f_{ij}] \in M_{m,n}(CB(Y_1,Y_2))$    with the map
$[y_j] \mapsto [\sum_j \, f_{ij}(y_j)]$.   
The next result asserts that these norms on $M_{m,n}(\mathbb{B}(Y_1,Y_2))$ are the same.  We will write this norm
as $\Vert [f_{ij}] \Vert_{cb}$.

\begin{lemma} \label{lll}   {\rm (\cite[Corollary 3.6]{BK2}}) \ Suppose that $Y_1$ and $Y_2$ are right $w^*$-rigged modules over 
a dual operator algebra $M$.    For each $m, n \in \mathbb{N}$ we have
$M_{m,n}(\mathbb{B}(Y,Z)) \cong
\mathbb{B}(C_n(Y),C_m(Z))$ completely isometrically.  
\end{lemma}

The interior tensor
product of $w^*$-rigged modules is functorial: 

\begin{proposition} \label{funct}
 Suppose that $Y_1$ and $Y_2$ are right $w^*$-rigged modules over 
a dual operator algebra $M$,
that $Z_1$ and $Z_2$ are right $M$-$N$-correspondences  
for a dual operator algebra 
$N$.   Write the left action
on $Z_k$ (abusively) as $\theta : M \to \mathbb{B}(Z_k)_N$.  
If $f = [f_{ij}] \in M_{m,n}(\mathbb{B}(Y_1,Y_2)_M)$, and
if $g = [g_{kl}] \in M_{p,q}(\mathbb{B}(Z_1,Z_2)_N)$  is a matrix of
adjointables which are also left $M$-module maps,
write $f \otimes g$ for  $[f_{ij} \otimes g_{kl}] \in M_{mp,nq}(\mathbb{B}(Y_1 \otimes_{\theta} Z_1,
Y_2 \otimes_{\theta} Z_2)$.  
 Then ${\lVert f \otimes g \rVert}_{cb} \leq {\lVert f \rVert}_{cb} {\lVert g \rVert}_{cb}$, where the `subscript cb' refers to the norm discussed
above the theorem.  Further, $\widetilde{f \otimes g} = \tilde{g} \otimes \tilde{f}$
for any $f \in \mathbb{B}(Y_1,Y_2)_M$ and $g \in \mathbb{B}(Z_1,Z_2)_N$.
\end{proposition}

\begin{proof}  Suppose that
 $f \in \mathbb{B}(Y_1,Y_2)_M$ and $g \in \mathbb{B}(Z_1,Z_2)_N$.
Since $g :Z_1 \to Z_2$ is a left $M$-module map, $\tilde{g}$ is a right $M$-module map.
Thus we can define $f \otimes^{\sigma h}_{M} g : Y_1 \otimes^{\sigma h}_{M} Z_1 \to Y_2 \otimes^{\sigma h}_{M} Z_2$,
and $\tilde{g} \otimes^{\sigma h}_{M} \tilde{f} : \tilde{Z_2} \otimes^{\sigma h}_{M} \tilde{Y_2} \to
\tilde{Z_1} \otimes^{\sigma h}_{M} \tilde{Y_1}$. By Corollary 2.4 in \cite{BK1},
$f \otimes g : Y_1 \otimes_{\theta} Z_1 \to Y_2 \otimes_{\theta} Z_2$ is a completely bounded
weak* continuous right $N$-module map, with ${\lVert f \otimes g \rVert}_{cb} \leq {\lVert f \rVert}_{cb} {\lVert g \rVert}_{cb}$.
That $f \otimes g$ is adjointable follows from the fact that it is weak* continuous (by the remark  a couple of paragraphs 
above Definition \ref{w*Mor}).   Alternatively, 
for $y \in Y_1, z \in Z_1, w \in \tilde{Z_2}, x \in \tilde{Y_2}$ we have
\begin{eqnarray*}
(w \otimes x, f(y) \otimes g(z))_N  &=& (w,\theta((x,f(y))_M)g(z))_N \\
&=& (w\theta((\tilde{f}(x),y)_M),g(z))_N  \\
 &=& (\tilde{g}(w \theta((\tilde{f}(x),y)_M)),z)_N \\ 
&=& ( \tilde{g}(w), \theta((\tilde{f}(x),y)_M) z)_N \\
 &=&  (\tilde{g}(w) \otimes \tilde{f}(x), y \otimes z)_N,
\end{eqnarray*} which also yields  the last statement.  

Next, let $f = [f_{ij}], g = [g_{kl}]$ be as in the statement.
By a two step method we may assume that $f = I$ or $g = I$.
If $g = I$ the norm inequality we want
 follows from  the case in the last paragraph,
Lemma \ref{lll},  and because
$$M_{m,n}(\mathbb{B}(Y_1 \otimes_{\theta} Z_1,
Y_2 \otimes_{\theta} Z_2)) \cong 
\mathbb{B}(C_{n}(Y_1 \otimes_{\theta} Z_1) , C_{m}(Y_2 \otimes_{\theta} Z_2))$$
which may be viewed as  $\mathbb{B}(C_{n}(Y_1) \otimes_{\theta} Z_1 ,
C_{m}(Y_2) \otimes_{\theta} Z_2)$.   Thus $[f_{ij} \otimes 1]$ may be regarded as the map
$h \otimes I$ on $C_{n}(Y_1) \otimes_{\theta} Z_1$, 
where $h$ is the map $C_{n}(Y_1) \to C_{n}(Y_2)$ associated with 
$[f_{ij}]$ as in the discussion above Lemma \ref{lll}.  Thus 
$$\Vert  [f_{ij} \otimes 1] \Vert_{cb} = \Vert h \otimes I \Vert_{\mathbb{B}(C_{n}(Y_1) \otimes_{\theta} Z_1 ,
C_{m}(Y_2) \otimes_{\theta} Z_2)} \leq \Vert h \Vert_{cb} = \Vert f \Vert_{cb}.$$
 If $f = I$,
the norm inequality we want follows by a standard trick (which could also have been used to give an alternative 
proof of the 
previous computation).
If $y = [y_{ij}] \in M_p(Y)$ and $z = [z_{ij}] \in M_p(Z)$
then  $$\sum_r \, y_{ir} \otimes g_{kl}(z_{rj})
= \lim_\alpha \, \sum_{r,k} \, y_{ir} \otimes g_{kl}(z^\alpha_k) (w^\alpha_k, z_{rj})$$
where $\phi_\alpha(z) = [(w^\alpha_k, z)]$ and $\psi_\alpha([b_k])
= \sum_k \, z^\alpha_k b_k$ are the maps in Definition \ref{wrig}, but for $Z$ in place of $Y$.
It follows that the norm of $[\sum_r \, y_{ir} \otimes g_{kl}(z_{rj})]$
is dominated by $\sup_\alpha \, \Vert [g_{kl}] \Vert_{cb}  
\Vert y \Vert \Vert z \Vert$.   That 
$\Vert I \otimes g \Vert_{cb} \leq
\Vert g \Vert_{cb}$  follows easily from this  if
we use \cite[Corollary 2.8]{BK1}.      
  \end{proof}

{\bf Remark.}   A similar result holds for rigged modules over approximately unital  operator algebras.
These matricial versions of the functoriality of the tensor product will be used in \cite{DBr}.

\medskip

The $w^*$-rigged interior  tensor product is  {\em    projective}.   This is because of the following:

\begin{proposition} \label{ispr}   The  normal module Haagerup tensor product is  projective:
If $u : Y_1 \to Y_2$ is a weak* continuous $M$-module complete quotient map between right 
dual operator $M$-modules, and $v : Z_1  \to Y_2$ is a weak* continuous $M$-module complete quotient map between left
dual operator $M$-modules, then $u \otimes v : Y_1 \otimes_{M}^{\sigma h} Z_1 \to 
 Y_2 \otimes_{M}^{\sigma h} Z_2$ is a  weak* continuous complete quotient map.
\end{proposition}   \begin{proof}     By functoriality of $\otimes_{M}^{\sigma h}$ (see \cite[Corollary 2.4]{BK1})
we have that $u \otimes v$ is a weak* continuous complete contraction.  
 If $z \in {\rm Ball}(Y_2 \otimes_{M}^{\sigma h} Z_2)$, then by  \cite[Corollary 2.8]{BK1}
$z$ is weak* approximable by a net $z_t \in 
{\rm Ball}(Y_2 \otimes_{hM} Z_2)$.   We may assume that $\Vert z_t \Vert < 1$ for all $t$.
By projectivity of $\otimes_{hM}$ (or of $\otimes_{h}$)
there exist $w_t \in {\rm Ball}(Y_1 \otimes_{hM} Z_1)$ with $(u \otimes v) (w_t) = z_t$.
Suppose that $w_{t_\nu} \to w \in {\rm Ball}(Y_1 \otimes_{M}^{\sigma h} Z_1)$, then
clearly $(u \otimes v) (w) = z$.   So $u \otimes v$ is a quotient map.
A similar argument works at the matrix levels using \cite[Corollary 2.8]{BK1}.
\end{proof}

Unlike the module Haagerup tensor product over a $C^*$-algebra  (see \cite[Theorem 3.6.5 (2)]{DBbook}), the  normal module Haagerup tensor product $\otimes_{M}^{\sigma h}$ need not be `injective'
for general dual operator modules if $M$ is a $W^*$-algebra.    However
by functoriality it is easy to see that we will have such injectivity for this tensor product for  $w^*$-orthogonally 
complemented submodules (see  the last section of \cite{DBr}) of 
$w^*$-rigged modules, even if $M$ is a dual operator algebra.    For example if 
$Y$ is a  $w^*$-orthogonally  
complemented submodule of a $w^*$-rigged module $W$ over $M$, and if $i$ and $P$ are the associated inclusion 
and projection maps, and if $Z$ is a  right $M$-$N$-correspondence, then  $P \otimes I_Z$ and $i  \otimes I_Z$
are completely contractive adjointable maps composing to the identity on $Y \otimes_\theta Z$.   So 
$Y \otimes_\theta Z$ is weak* homeomorphically completely isometrically $M$-module isomorphic to a $w^*$-orthogonally 
complemented submodule of $W \otimes_\theta Z$.   A similar statement may be made 
for an appropriately complemented $M$-$N$-`subcorrespondence' of $Z$.

This injectivity of the weak* interior tensor product 
will work for any weak* closed submodules of $w^*$-rigged modules
 over a $W^*$-algebra (that is, $W^*$-modules), because such
are automatically $w^*$-orthogonally
complemented \cite{Bsd}.     
It follows that the weak* interior tensor with a right correspondence
is `exact' on the category of $W^*$-modules.
Thus given an exact sequence of $W^*$-modules over a  $W^*$-algebra $M$, 
$$0 \; \longrightarrow \; D  \longrightarrow \; E
 \longrightarrow \; F  \longrightarrow \; 0$$
with the first of these adjointable morphisms  completely isometric and  the second 
a complete quotient map, and given a right $M$-$N$-correspondence $Z$, we get 
an exact sequence 
$$0 \; \longrightarrow \; D  \otimes_\theta Z \; \longrightarrow \; E \otimes_\theta Z
 \longrightarrow \; F  \otimes_\theta Z \longrightarrow \; 0$$
of the same kind.   Indeed this all follows from
 the `commutation with direct sums'
property at the end of \cite[Section 3]{BK2}.   It might be interesting to investigate such exactness
of the interior tensor with a correspondence on the category of $w^*$-rigged modules.  


\subsection{HOM-tensor relations}    (See \cite[Theorem 3.6]{Bsd} for the self-adjoint variant of these.)  In our context, the  HOM spaces will be the spaces $\mathbb{B}( - , - )$ of all weak* continuous completely bounded  module maps.

\begin{lemma} \label{sdrel}  If $Y$ is a right $w^*$-rigged module over $M$,  and $Z$ is a left (resp.\ right) dual operator module over $M$, then $Y \otimes^{\sigma h}_M Z \cong \mathbb{B}_M(\tilde{Y},Z)$
(resp.\  $Z  \otimes^{\sigma h}_M \tilde{Y} \cong \mathbb{B}(Y,Z)_M$) 
completely isometrically and weak* homeomorphically.
\end{lemma}   \begin{proof}   The respectively case is  \cite[Theorem 3.5]{BK2}.  
By the `other-handed variant' of  \cite[Theorem 3.5]{BK2} we have $CB^\sigma_M(\tilde{Y},Z)
\cong \widetilde{\tilde{Y}} \otimes^{\sigma h}_M Z = Y \otimes^{\sigma h}_M Z$.   
\end{proof}

\begin{theorem}
Let $M$ and $N$ be dual operator algebras. We have the following completely isometric identifications:
\begin{enumerate}

\item  $\mathbb{B}(Y \otimes_{\theta} Z, W)_{N}  \cong   \mathbb{B} ( Y, \mathbb{B}(Z,W)_{N} )_{M}$,  where $Y$ is a right $w^*$-rigged module over $M$, 
$Z$ is a right  $M$-$N$ correspondence,  and $W$ is a right dual operator module over $N$.

\item  $\mathbb{B}(Y, (Z \otimes^{\sigma h}_{N} W))_{M} \cong Z \otimes^{\sigma h}_{N} \mathbb{B} (Y, W) _{M} $, where $Y$ is a right $w^*$-rigged module over $M$, $W$ is a dual operator $N$-$M$-bimodule and $Z$ is a right dual operator $N$-module.

\item  $\mathbb{B}_N (\mathbb{B}(Y, W)_{M}, X) \cong Y \otimes^{\sigma h}_{M} \mathbb{B}_N(W, X)$, where $Y$ is a right $w^*$-rigged module over $M$, $X$ is a dual left $N$-operator module, and $W$ is a left  $N$-$M$-correspondence.

\item $ \mathbb{B}_M (X, \mathbb{B} (Z, W)_{N}) \cong \mathbb{B}(Z, \mathbb{B}_M(X, W))_{N}$ where $X$, $Z$ are left and right $w^*$-rigged modules over $M$ and $N$ respectively, and $W$ is a dual  operator 
$M$-$N$-bimodule. \end{enumerate}
\end{theorem}

\begin{proof}  
The proofs all follow from Lemma \ref{sdrel}, and  Corollary 3.3 in \cite {BK2}, and the associativity of the normal module Haagerup tensor product, and the fact that $\widetilde{\tilde{Y}} = Y$ for  $w^*$-rigged modules.   Since the proofs are all similar we just prove a couple of them.  For (1) note that 
   $$\mathbb{B}_N(Y \otimes_{\theta} Z, W)  \cong   W \otimes^{\sigma h}_{N}  \widetilde{Y \otimes_{\theta} Z}  \cong  W \otimes^{\sigma h}_{N} ( \tilde{Z} \otimes ^{\sigma h}_{M} \tilde{Y})   \cong      (W \otimes^{\sigma h}_{N}  \tilde{Z} ) \otimes ^{\sigma h}_{M} \tilde{Y}  $$
which is isomorphic to $\mathbb{B} (Z, W)_{N}  \otimes ^{\sigma h}_{M} \tilde{Y}
\cong  \mathbb{B} ( Y, \mathbb{B}(Z,W)_{N} )_{M}$.
For (3), 
$$\mathbb{B}_N (\mathbb{B}(Y, W)_{M}, X) \cong  \mathbb{B}_N (W \otimes^{\sigma h}_{M} \tilde{Y}, X) \cong \widetilde{(W \otimes^{\sigma h}_{M}  \tilde{Y})} \otimes^{\sigma h} _{N} X  \cong  (Y \otimes^{\sigma h}_{M} \tilde{W} ) \otimes^{\sigma h}_{N} X$$
which is isomorphic to $Y \otimes^{\sigma h}_{M} ( \tilde{W} \otimes^{\sigma h}_{N} X)    \cong Y \otimes^{\sigma h}_{M} \mathbb{B}_N(W, X)$.  
\end{proof}

\section{The Picard group}   
 
We now discuss the Picard group of a dual operator algebra, following the route in \cite{BJ}. 
Throughout this section $A$ will be a dual operator algebra.  We define the {\em weak Picard group} of $A$, denoted by Pic$_w(A)$,
to be the collection of all $A$-$A$-bimodules implementing
a weak* Morita equivalence of $A$ with itself, with two such bimodules
identified if they are completely isometrically isomorphic
and weak* homeomorphic via an $A$-$A$-bimodule map.  The multiplication on
Pic$_w(A)$ is given by
the module normal Haagerup tensor product $\otimes^{\sigma h}_{A}$.

Any weak*  continuous completely isometric 
automorphism $\theta$ of $A$ defines a weak* Morita
equivalence $A$-$A$-bimodule $A_\theta$ by `change of rings' on the right.   This is just $A$ with the usual left module action, and with
right module action $x \cdot a$ = $x \theta(a)$.

\begin{lemma} \label{t}
The bimodule $A_{\theta}$ above is a weak* Morita equivalence bimodule for $A$,
 with `inverse bimodule' $A_{\theta^{-1}}$.
\end{lemma}

\begin{proof}  This follows using Definition \ref{bk1} and Lemma \ref{lebe} below, but we will give 
a more explicit proof.  
 If $A$ is a dual operator algebra, then
we prove that $(A, A, A_{\theta}, A_{\theta^{-1}})$ is a weak* Morita context
using Theorem 3.3 in \cite{BK1}.  Define a pairing
$( \cdot ) : A_{\theta} \times A_{\theta^{-1}} \to A$ taking
$(a, a') \mapsto a\theta(a')$.
It is easy to check that $(\cdot )$ is a separately
$w^*$-continuous completely contractive $A$-bimodule map
which is balanced over $A$.
Similarly, define another pairing
$[ \cdot ] : A_{\theta^{-1}} \times A_{\theta} \to A$ taking
$[a, a'] \mapsto a \theta^{-1}(a')$.
Again it is easy to check that
$[ \cdot ]$ is a separately $w^*$-continuous completely
contractive $A$-module map which is 
balanced over $A$.   It is simple algebra to check that  
$(x,y)x' = x [y,x']$ and $y'(x, y) = [y', x] y$.
Checking the last assertion of Theorem 3.3 in \cite{BK1} is also obvious.
 Hence $(A, A, A_{\theta}, A_{\theta^{-1}})$
is a weak* Morita context.
\end{proof}

Let Aut$(A)$ denote the group of weak* continuous
completely isometric automorphisms of $A$.
For $\alpha$, $\beta \in $Aut$(A)$
let $_{\alpha} A _{\beta}$ denote 
$A$ viewed as an $A$-$A$-bimodule with the left action $a \cdot x = \alpha(a) x$
 and right action $x \cdot b = x \beta(b)$.

\begin{lemma} \label{u}
If $\alpha, \beta, \gamma \in $Aut$(A)$ then,
$_{\alpha}A_{\beta}$  $\cong$ $ _{\gamma \alpha} A_{\gamma \beta}$ 
completely $A$-$A$-isometrically.
\end{lemma}

\begin{proof}
The map  $\gamma$ is the required isomorphism.
\end{proof}

\begin{lemma} \label{lebe}
For $\theta_1, \theta_2 \in $Aut$(A)$,
$A_{\theta_1} \otimes^{\sigma h}_{A} A_{\theta_2}$
$\cong$ $A_{\theta_1\theta_2}$ completely $A$-$A$-isometrically.
\end{lemma}

\begin{proof}
From Lemma \ref{u},
$A_{\theta_1} \otimes^{\sigma h}_{A} A_{\theta_2} $
$\cong$ $_{\theta_1^{-1} } A \otimes^{\sigma h}_{A} A_{\theta_2}$
$\cong$ $_{\theta_1^{-1}} A_{\theta_2}$ $\cong$ $A_{\theta_1 \theta_2}$.
 \end{proof}

\begin{proposition}
The collection $\{ A_{\theta} : \theta \in {\rm Aut}(A) \}$  constitutes a subgroup
of {\rm Pic}$_{w}(A)$, which is isomorphic
to the group Aut$(A)$ of weak* continuous completely
 isometric automorphisms of $A$.
\end{proposition}
\begin{proof}
This follows from the above lemma.
\end{proof}

If $X$ is a weak* Morita equivalence $A$-$A$-bimodule,
and if $\theta \in $Aut$(A)$, then let $X_{\theta}$ be 
$X$ with the same left module action, but with right module action 
changed to $x \cdot a = x \theta(a)$.  
  
\begin{lemma}
If $X$ is a weak* Morita equivalence bimodule,
then $X_{\theta} \cong X \otimes^{\sigma h}_{A} A_{\theta}$
completely $A$-$A$-isometrically.
Also,  $X_{\theta}$ is a weak*  Morita equivalence
bimodule.
\end{lemma}

\begin{proof}  
The module action $(x, a) \mapsto x a$ is a completely contractive, separately weak* continuous 
balanced bilinear map, so by the universal property of the normal module Haagerup tensor product it induces a completely contractive weak* continuous linear map $m :  X \otimes^{\sigma h}_{A} A_{\theta}  \to  X_{\theta}$. 
There is a completely contractive inverse map $x \mapsto x \otimes 1$, so that  
 $m$ is a surjective complete isometry,
and it is easily seen to be an $A$-$A$-bimodule map. If $(A, A, X, Y)$ is a weak Morita context as in Definition \ref{w*Mor} then by Lemma \ref{t}, and properties of the normal module Haagerup tensor product, $(A, A, X_{\theta},  _{\theta^{-1}}Y)$  is a Morita context.  To see this, note that by  associativity of the normal module Haagerup tensor product, and some of the lemmas above in the present section, 
 $$(X \otimes^{\sigma h}_{A} A_{\theta})  \otimes^{\sigma h}_{A} ( A_{\theta ^{-1}}  \otimes^{\sigma h}_{A} Y) \cong X \otimes^{\sigma h}_{A}  A \otimes^{\sigma h}_{A} Y \cong  X \otimes^{\sigma h}_{A} Y  \cong A$$ completely isometrically and weak*-homeomorphically.
Similarly 
 $$(A_{\theta ^{-1}}   \otimes^{\sigma h}_{A} Y) \otimes^{\sigma h}_{A}  (X \otimes^{\sigma h}_{A}  A_{\theta})  \cong   A_{\theta ^{-1}}   \otimes^{\sigma h}_{A}  A \otimes^{\sigma h}_{A} A_{\theta} \cong  \, A_{\theta ^{-1}}   \otimes^{\sigma h}_{A} A_{\theta} \cong A$$ 
completely isometrically and weak*-homeomorphically.  
\end{proof}

An $A$-$A$-bimodule $X$ will be called `symmetric'
if $ax = xa$ for all $a \in A$, $x \in X$.

\begin{proposition} \label{gold}  
 For a  commutative dual operator algebra $A$,  ${\rm Pic}_{w}(A)$ is a
semidirect product of ${\rm Aut}(A)$ and the subgroup of
${\rm Pic}_{w}(A)$ consisting of symmetric equivalence
bimodules. Thus, every $V \in {\rm Pic}_{w}(A)$ equals $X_{\theta}$,
for a symmetric $X \in {\rm Pic}_{w}(A)$, and 
some $\theta \in {\rm Aut}(A)$.
\end{proposition}

\begin{proof} 
Suppose that $X$ is any weak* Morita equivalence
$A$-$A$-bimodule. Then any $w^*$-continuous right $A$-module map $T
: X \to X$ is simply left multiplication by a fixed element of $A$.
In fact  we have $A \cong
CB^{\sigma}(X)_A$, via a map $L : A \to CB(X)$  
(see e.g.\ Theorem 3.6 in \cite{BK1}). For fixed $a \in A$, the map $x \mapsto xa$ on
$X$, is a $w^*$-continuous completely bounded $A$-module map with
completely bounded norm = $\lVert a \rVert$. Hence by the above
identification, there exists a unique $a' \in A$ such that $a' x = x
a$ for all $x \in X$ and $\lVert a' \rVert = \lVert a \rVert$.
  Define $\theta(a) = a'$, then we claim that 
$\theta$ is a weak* continuous completely 
isometric unital automorphism of $A$.  To see that $\theta$ is a homomorphism,
  let $a_1, a_2 \in A$ and let
  $T, S$ and $U$ be maps from $X$ to $X$ simply
  given by right multiplication
  with $a_1, a_2$ and $a_1a_2$ respectively.
  Let $\theta(a_1) = a_1'$,
  $\theta(a_2) = a_2'$ and $\theta(a_1a_2) = a_3$.
  Since $U = ST$, we have $L(U) = L(T) L(S)$
  (recall $L$ is an anti-homomorphism,   see e.g.\ Theorem 3.6 in \cite{BK1}).  This implies $a_3 = a_1' a_2'$, that is,  $\theta(a_1a_2) = \theta(a_1) \theta(a_2)$.
 
 Note that $$\Vert [ \theta(a_{ij}) x_{kl} ] \Vert = \Vert [ x_{kl} a_{ij} ] \Vert \leq \Vert [ x_{kl} ] \Vert
 \Vert [a_{ij} ] \Vert,$$
and so using the isomorphism  $A \cong
CB^{\sigma}_{A}(X)$ above we see that $\theta$ is completely contractive.  
That $\theta$ is completely 
isometric follows (e.g.\ by a similar argument for $\theta^{-1}$). 
For the weak* continuity, note that 
if we have a bounded net $a_t \buildrel  w^* \over \to a $ then
$x a_t \buildrel  w^* \over \to xa$ for all $x \in X$.
Since the map $L$ above is a weak* homeomorphism, 
we deduce that $\theta(a_t) \buildrel  w^* \over \to \theta(a)$.   
  That $\theta$ is surjective follows by symmetry.
  
  Thus we have defined a surjective group homomorphism
  $Pic_{w}(A) \to Aut(A)$
 taking $X \mapsto \theta$.
  To see that this does define a group homomorphism,
  let $X, Y \in Pic_{w}(A)$.  Let
  $X \otimes^{\sigma h}_{A} Y \mapsto \theta$
  and $X \mapsto \theta_1, Y \mapsto \theta_2$
  under the above identification.
   We need to show that $\theta = \theta_1 \theta_2$.
   Let $ a \in A $ and
   $\theta(a) = a'$, $\theta_2(a) = a_2$ and $\theta_1(a_2) = a_1$.
We need to show that $a' = a_1$. Consider a rank one tensor $x
\otimes y \in X \otimes^{\sigma h}_{A} Y$ for $x \in X$ and $y \in
Y$. Then $a'x \otimes y = x \otimes ya$, $a_1x = xa_2$ and $a_2y =
ya$. We have $a'(x \otimes y) = a'x \otimes y$ and $$a_1(x \otimes
y) = a_1x \otimes y = xa_2 \otimes y = x \otimes a_2 y = x \otimes
ya.$$ Since finite rank tensors are weak* dense in $X
\otimes^{\sigma h}_{A} Y$, we have $a'z = a_1 z$ for all $z \in X
\otimes^{\sigma h}_{A}Y$. This implies $a' = a_1$ which proves the
required assertion.

  From Lemma \ref{t}, for $\theta \in Aut(A)$ we have
  $A_{\theta} \in Pic_{w}(A)$, hence the above
  homomorphism has a 1-sided inverse $Aut(A) \to Pic_{w}(A)$.
 The above homomorphism restricted
 to modules  of the form $A_{\theta}$ for $\theta \in Aut(A)$ is the identity map.
 That is,  for $\theta \in Aut(A)$   the above homomorphism
 takes the weak*  Morita equivalence
 bimodule $A_{\theta}$ to $\theta$. 
 Moreover the kernel of the homomorphism equals the 
symmetric  equivalence bimodules. This proves the `semidirect
product' statement. For the last assertion, note that $X =
(X_{\theta^{-1}})_{\theta}$. From the above, $x \theta^{-1}(a) =
\theta(\theta^{-1}(a)) x = ax$,   which proves that
$X_{\theta^{-1}}$ is symmetric.
\end{proof}

{\bf Remark.}   Similar results will hold for strong Morita equivalence bimodules in the sense of
\cite{BMP} over a norm closed operator algebra $A$, and their associated Picard group.

\bigskip

Thus we may assume henceforth that $X$ is symmetric, if $A$ is a commutative dual operator algebra.

\begin{proposition} \label{si}
 Suppose that we have a weak* Morita context
$(A, A, X, Y)$, with $A$ a weak* closed subalgebra of a
commutative von Neumann algebra $M$.  Suppose that $A$ generates
$M$ as a von Neumann algebra.  Then every
symmetric weak equivalence $A$-$A$-bimodule is completely  isometrically
$A$-$A$-isomorphic to a weak* closed $A$-$A$-subbimodule of $M$.
\end{proposition}

\begin{proof} 
From Theorem 5.5 in \cite{BK1}, $X$
dilates to a weak Morita equivalence $M$-$M$-bimodule
$W = M \otimes^{\sigma h}_{A} X$.
Let $(M, M, W, Z)$
be the corresponding $W^*$-Morita context.
From Theorem 3.5 in \cite{UK1}, $W$ contains $X$
completely isometrically as a weak* closed $A$-submodule; and indeed it is clear that this
is as a sub-bimodule over $A$.  
It is helpful to consider the inclusion
$$\left[
\begin{array}{ccl}
A & X \\
Y & A 
\end{array}
\right]   \subset    \left[
\begin{array}{ccl}
M & W \\
Z & M  
\end{array}
\right] $$
\noindent of linking algebras.
If $X$ is symmetric, then since $W = M \otimes^{\sigma h}_{A} X$
we have $wa = aw$ for all $w \in W, a \in A$.
Similarly for $Z  \cong Y \otimes^{\sigma h}_{A} M$,
 we have $z a = a z$. Since  $Z = W^{*}$, we have $wm^* = (mw^*)^*
= (w^*m)^* = m^* w$.
Therefore $xw = wx$ for all $w \in W$, $x \in M$.
Thus $W$ is a symmetric element of Pic$_{w}(M)$.

If  $M$ is a commutative von Neumann algebra, then it is well known that 
the Picard group of $M$  is just Aut$(M)$,
 and $M$ is the only symmetric element of Pic$_{w}(M)$.  We include a quick proof of this for the reader's 
 convenience.   Indeed suppose that  $Z$ is a
symmetric weak equivalence $M$-$M$-bimodule.
 Suppose that  $M$ is a von Neumann
algebra in $B(H)$, and let $K = Z \otimes_\theta H$, the induced representation.
Then $Z \subset B(H, K)$ is a WTRO, which is commutative in the sense of
\cite[Proposition 8.6.5]{DBbook}, and hence (see e.g.\ the proof of the last cited
result) $Z$ contains a unitary $u$ with $u u^* = I_K$ and $u^* u = I_H$.
Also, the map $R : z \mapsto u^* z$ is a completely isometric right $M$-module map
from $Z$ onto $M$.  That is, $Z = u M$.  Note that if $\theta : M \to B(Z)$ is
the left action
of $M$ on $Z$, then since $Z$ is symmetric we have $\theta(a) (u b) = u b a = u a u^* (u b)$,
for $a,b \in M$.   That is, $\theta(a)$ corresponds to $a \mapsto u a u^*
\in B(K)$.    Then $R(\theta(a) z) = u^* u a u^* z = a R(z)$, so that
$R$ is a bimodule map.  That is, $Z \cong M$ as equivalence $M$-$M$-bimodules.
Hence the Picard group of $M$  is just Aut$(M)$.

Putting the last two paragraphs together, we have proved that every  
symmetric $A$-$A$-bimodule `is' a weak* closed $A$-$A$-subbimodule of $M$.
\end{proof}   
 
\begin{corollary}  The weak Picard group of $H^\infty(\Ddb)$ is isomorphic 
to the group of conformal automorphisms of the disk.
\end{corollary} 

\begin{proof}  Since the monomial $z$ generates $A = H^\infty(\Ddb)$ as a dual algebra,
any automorphism $\theta$ of $A$ defines a map $\tau$ on the disk
by $\theta(f)(w) = f(\tau(w))$, so that $\tau = \theta(z) \in
H^\infty$.  By looking at $\theta^{-1}$, it is easy to see that
 $\tau$ is a conformal map.  Thus the group of conformal automorphisms of the disk
is isomorphic
to the group Aut$(A)$.
We will be done by Proposition \ref{gold}, if
 we can show that every symmetric equivalence $A$-$A$-bimodule
is $A$-$A$-isomorphic to $A$.

Let $M = L^\infty(\Tdb)$, and let $X$ be a 
symmetric equivalence $A$-$A$-bimodule.
By Proposition \ref{si},
$X$ can be taken to be a weak* closed  $A$-submodule of $M$.   By Beurling's theorem
it follows that $X$ is singly generated.    Indeed there is a function $k \in X$,
 which is either a projection or a unitary in $M$, with $X$ equal to the 
weak* closure of $Mk$ or $A k$ respectively \cite{Gam}.  If $k$ were a nontrivial projection $p \in M$, then $X M \subset k M$,
so that $X$ does not generate $Z$ as an $M$-module, which is a 
contradiction.   So $k$ is unitary, and hence, as in the proof of
Proposition \ref{si}, $X \cong A$ as dual operator $A$-$A$-bimodules. 
\end{proof}

 We thank Paul Muhly for a  discussion on the proof of the last result.

\end{document}